\documentclass[a4paper,10pt,article]{amsart}
\usepackage{amsmath,amscd,amssymb,amsthm,verbatim,indentfirst,dsfont,mathrsfs,ipa,graphicx,textcomp,enumerate}
\usepackage[all]{xy}
\usepackage{ipa,booktabs,longtable,xypic}
\usepackage[numbers]{natbib}
\usepackage{commutative-diagrams}
\usepackage[colorlinks=true,
            linkcolor=blue,
            anchorcolor=blue,
            citecolor=blue]{hyperref}

\newtheorem{theorem}{Theorem}[section]
\newtheorem{lemma}[theorem]{Lemma}
\newtheorem{corollary}[theorem]{Corollary}
\newtheorem{proposition}[theorem]{Proposition}
\newtheorem{conjecture}[theorem]{Conjecture}

\theoremstyle{definition}
\newtheorem{definition}[theorem]{Definition}

\theoremstyle{remark}
\newtheorem{remark}[theorem]{Remark}

\numberwithin{equation}{section}

\DeclareMathOperator{\supp}{supp}
\DeclareMathOperator{\prop}{prop}

\title{A localization algebra approach to the Baum-Connes conjecture for extensions}
\begin{document}

\author{Jianguo Zhang}
\address{School of Mathematics and Statistics, Shaanxi Normal University, Xi’an 710119, China.}
\email{jgzhang@snnu.edu.cn}

\thanks{This work was supported by NSFC (Nos. 12171156, 12271165, 12301154) and the Fundamental Research Funds for the Central University (No. 1301032574).}
\date{\today}

\begin{abstract}
For an extension $1\rightarrow N \rightarrow \Gamma \xrightarrow{q} \Gamma / N \rightarrow 1$ of discrete countable groups, it is known that the Baum-Connes conjecture with coefficients holds for $\Gamma$ if it holds for $\Gamma / N$ and $q^{-1}(F)$ for any finite subgroup $F$ of $\Gamma / N$. In this paper, we employ a localization algebra approach to prove this result.
\end{abstract}
\pagestyle{plain}
\maketitle


\section{Introduction}
Let $\Gamma$ be a discrete countable group and $A$ be a $C^{\ast}$-algebra equipped with a $\Gamma$-action. The Baum-Connes conjecture with coefficients in $A$ for $\Gamma$ asserts that the assembly map $\mu: K^{\Gamma}_{\ast}(\underline{E}\Gamma; A) \rightarrow K_{\ast}(A\rtimes_r \Gamma)$ is an isomorphism (cf. \cite{BC-birth}\cite{BCH-1994}). The left-hand side of the conjecture is the equivariant $K$-homology with coefficients in $A$ of $\underline{E}\Gamma$ which is the classifying space for proper actions of $\Gamma$. And the right-hand side of the conjecture is the $K$-theory of the reduced crossed product, which provides a receptacle for higher indices of $\Gamma$-equivariant elliptic differential operators on a Riemannian manifold equipped with a proper, co-compact isometric $\Gamma$-action. Thus, the Baum-Connes conjecture with coefficients furnishes a topological formula to compute higher indices, and hence it is a generalization of the Atiyah-Singer index theorem. Moreover, the Baum-Connes conjecture implies the Novikov conjecture, the Gromov-Lawson-Rosenberg conjecture and the Kaplansky-Kadison conjecture (cf. \cite{PPA-BC-survey}\cite{WillettYu-Book}).

The Baum-Connes conjecture with coefficients had been verified for a large class of groups, for examples, a-T-menable groups and hyperbolic groups (cf. \cite{Higson-Kasparov}\cite{Lafforgue-2002}\cite{Lafforgue-2012}\cite{MY-BC-hyperbolic}). On the other hand, counterexamples to the Baum-Connes conjecture with non-trivial coefficients do exist (cf. \cite{HLS-2002}). Besides, the Baum-Connes conjecture is still open for special linear groups $SL_{n}(\mathbb{Z})$ for $n\geq 3$, although the assembly map $\mu$ is injective for these groups (cf. \cite{GHW-2005}).

In order to enlarge the class of groups satisfying the Baum-Connes conjecture with coefficients, we need to consider if the conjecture is closed under group extensions. In \cite[Theorem 3.1]{Oyono-BC-extensions}, Oyono-Oyono proved the following theorem.

\begin{theorem}\label{OO-thm}
	Let $1\rightarrow N \rightarrow \Gamma \rightarrow \Gamma / N \rightarrow 1$ be an extension of discrete countable groups. Assume that
	\begin{enumerate}
		\item \label{OO-thm-1} every subgroup of $\Gamma$ containing $N$ as a subgroup of finite index satisfies the Baum-Connes conjecture with coefficients;
		\item \label{OO-thm-2} the group $\Gamma / N$ satisfies the Baum-Connes conjecture with coefficients.
	\end{enumerate}
	Then the group $\Gamma$ satisfies the Baum-Connes conjecture with coefficients.
\end{theorem}
Similar results for extensions of locally compact, second countable groups were obtained in \cite{CEO-2004}, \cite{CE-Permanence-BC} and \cite{MN-2006}. And in \cite[Theorem 5.4]{Arano-Kubota-BCextensions}, Arano and Kubota refined the above result by replacing the condition (\ref{OO-thm-1}) of Theorem \ref{OO-thm} by the condition that every subgroup of $\Gamma$ which contains $N$ and has the finite cyclic quotient in $\Gamma/ N$ satisfies the Baum-Connes conjecture with coefficients.

In this paper, we employ a localization algebra approach to prove Theorem \ref{OO-thm} by the following three steps.

Step 1: we use equivariant Roe algebras and equivariant localization algebras to formulate the Baum-Connes conjecture with coefficients based on the work of Guentner, Willett and Yu in \cite{GWY-2024}(see Conjecture \ref{BCC}).

Step 2: we introduce the concept of equivariant localization algebras along one direction for group direct products (See Definition \ref{Def-Localg-along}), and use it to prove that the Baum-Connes conjecture is closed under direct products by the Mayer-Vietoris argument (see Theorem \ref{main-thm1}).

Step 3: we complete the proof of Theorem \ref{OO-thm} by translating the Baum-Connes conjecture with coefficients for extensions into the conjecture for direct products (see Section \ref{Sec-main-thm}).

The paper is organized as follows. In Section \ref{Sec-BC}, we recall some basic notions and properties of equivariant Roe algebras with coefficients and equivariant localization algebras with coefficients. And we also formulate the Baum-Connes conjecture with coefficients. In Section \ref{Sec-BCprod}, we first study equivariant Roe algebras for direct products. Then we define the concept of equivariant localization algebras along one direction for group direct products. Lastly, we verify that the Baum-Connes conjecture with coefficients is closed under direct products. Finally, we prove Theorem \ref{OO-thm} in Section \ref{Sec-main-thm}. 

\section*{Acknowledgements}
The author would like to thank Prof. Ralf Meyer and Prof. Shintaro Nishikawa for pointing out an essential mistake appeared in Theorem \ref{main-thm1} in the first version of the paper. 

\section{A description for the Baum-Connes conjecture with coefficients} \label{Sec-BC}
In this section, we will introduce the Baum-Connes conjecture with coefficients for discrete countable groups by using equivariant Roe algebras with coefficients and equivariant localization algebras with coefficients.

\subsection{Equivariant Roe algebras with coefficients}



Let $\Gamma$ be a discrete countable group and $X$ be a proper metric space (properness means that any bounded closed subset is compact). An isometric action of $\Gamma$ on $X$ is called to be \textit{proper}, if the set $\{\gamma\in \Gamma: \gamma K\cap K \neq \emptyset\}$ is finite for any compact subset $K$ in $X$. And an isometric action of $\Gamma$ on $X$ is called to be \textit{co-compact}, if the quotient space $X/\Gamma$ is compact.

\begin{definition}\label{Def-Gamma-C}
Let $\Gamma$ be a discrete countable group. 
  \begin{enumerate}
  	\item A proper metric space $X$ is called to be a \textit{$\Gamma$-space}, if there exists a proper, co-compact and isometric right $\Gamma$-action on $X$.
    \item A $C^{\ast}$-algebra $A$ is called to be a \textit{$\Gamma$-$C^{\ast}$-algebra}, if there exists a $\Gamma$-action on $A$ by $\ast$-automorphisms. 
  \end{enumerate}
\end{definition}

Let $\Gamma$ be a discrete countable group and $A$ be a $\Gamma$-$C^{\ast}$-algebra equipped with an action $\alpha$. Let $X$ be a $\Gamma$-space and $Z_X$ be a $\Gamma$-invariant countable dense subset in $X$. Let $H$ be a separable Hilbert space endowed with the trivial $\Gamma$-action.

Consider the following right Hilbert $A$-module
$$_{X}E_A=\ell^2(Z_X) \otimes A \otimes H \otimes \ell^2(\Gamma).$$
The right $A$-action on $_{X}E_A$ is defined by
$$(\xi \otimes a \otimes h \otimes \eta)\cdot a'=\xi \otimes aa' \otimes h \otimes \eta,$$
for any $a'\in A$. And the $A$-valued inner product on $_{X}E_A$ is defined to be 
$$\langle \xi_1 \otimes a_1 \otimes h_1 \otimes \eta_1, \xi_2 \otimes a_2 \otimes h_2 \otimes \eta_2 \rangle=\langle \xi_1, \xi_2 \rangle \langle h_1, h_2 \rangle \langle \eta_1, \eta_2 \rangle a_1^{\ast}a_2.$$
Define a $\Gamma$-action on $_{X}E_A$ by 
$$U_{\gamma}: _{X}E_A \rightarrow _{X}E_A, \:\: \delta_{x}\otimes a\otimes h\otimes \delta_{r'} \mapsto \delta_{x\gamma^{-1}}\otimes \alpha_{\gamma}(a)\otimes h\otimes \delta_{\gamma\gamma'},$$
for any $\gamma\in \Gamma$. Let $\mathcal{B}(X)$ be the algebra of all bounded Borel functions on $X$. Define a left action of $\mathcal{B}(X)$ on $_{X}E_A$ by
$$f\cdot (\xi\otimes a\otimes h\otimes \eta)=f\xi\otimes a\otimes h\otimes \eta,$$
for any $f\in \mathcal{B}(X)$. Then we have
$$U_{\gamma}aU_{\gamma^{-1}}=\alpha_{\gamma}(a)\:\:\text{and}\:\: U_{\gamma}fU_{\gamma^{-1}}=\gamma\cdot f,$$
for any $\gamma\in \Gamma$, $a\in A$ and $f\in \mathcal{B}(X)$, where $\gamma\cdot f(x)=f(x\gamma)$.

We denote by $\mathcal{L}(_{X}E_A)$ the algebra of all adjointable operators on $_{X}E_A$ and by $\mathcal{K}(_{X}E_A)$ the algebra of all compact operators on $_{X}E_A$. Then we have an isomorphism $\mathcal{K}(_{X}E_A)\cong \mathcal{K}(\ell^2(Z_X))\otimes \mathcal{K}(H)\otimes \mathcal{K}(\ell^2(\Gamma))\otimes A$. 

\begin{definition}\label{Def-prop}
	For $S\subseteq X$, let $\chi_{S}$ be an operator on $_{X}E_A$ defined by $\chi_{S}(\xi\otimes a\otimes h\otimes \eta)=\xi|_{S}\otimes a\otimes h\otimes \eta$. And let $T\in \mathcal{L}(_{X}E_A)$. 
	\begin{enumerate}
		\item The \textit{support} of $T$, denoted by $\supp(T)$, is defined to be
		$$\{(x,x')\in X\times X: \chi_{V} T \chi_{U}\neq 0 \:\: \text{for any open neighborhoods $U$ of $x$, $V$ of $x'$}\}.$$
		\item The \textit{propagation} of $T$, denoted by $\text{prop}(T)$, is defined to be
		$$\text{prop}(T)=\sup\{d(x,x'): (x,x')\in \supp(T)\}.$$  
		\item $T$ is called to be \textit{locally compact}, if $\chi_K T, T\chi_K\in \mathcal{K}(_{X}E_A)$ for any compact subset $K\subseteq X$.
		\item $T$ is called to be \textit{$\Gamma$-invariant}, if $U_{\gamma} T U_{\gamma^{-1}}=T$ for any $\gamma\in \Gamma$.
	\end{enumerate}
\end{definition}

\begin{remark}
	For any $T\in \mathcal{L}(_{X}E_A)$, we can represent $T$ as a matrix $(T_{x,x'})_{x,x'\in Z_X}$, where $T_{x,x'}=\chi_{x} T \chi_{x'}$. Then, $T$ has finite propagation if and only if there exists $R>0$ such that $T_{x,x'}=0$ for all $x,x'\in Z_X$ with $d(x,x')>R$. And if $T_{x,x'}=a_{x,x'}\otimes K_{x,x'} \otimes Q_{x,x'}$ for some $a_{x,x'}\in A$, $K_{x,x'}\in \mathcal{K}(H)$ and $Q_{x,x'}\in \mathcal{K}(\ell^2(\Gamma))$. Then $T$ is $\Gamma$-invariant if and only if $a_{x,x'}=\alpha_{\gamma}(a_{x\gamma, x'\gamma})$, $K_{x,x'}=K_{x\gamma, x'\gamma}$ and $Q_{x,x'}=W_{\gamma} Q_{x\gamma, x'\gamma} W_{\gamma^{-1}}$ for any $\gamma\in \Gamma$, where $W_{\gamma}: \ell^2(\Gamma) \rightarrow \ell^2(\Gamma),\: \delta_{\gamma'}\mapsto \delta_{\gamma\gamma'}$.
\end{remark}

\begin{definition}\label{Def-eqRoecoe}
	Let $\Gamma$ be a discrete countable group, $X$ be a $\Gamma$-space and $A$ be a $\Gamma$-$C^{\ast}$-algebra. Define $\mathbb{C}[X, A]^{\Gamma}$ to be the $\ast$-algebra consisting of all $\Gamma$-invariant, locally compact operators with finite propagation.
    The \textit{equivariant Roe algebra with coefficients} in $A$ of $X$, denoted by $C^{\ast}(X, A)^{\Gamma}$, is defined to be the norm closure of $\mathbb{C}[X, A]^{\Gamma}$ in $\mathcal{L}(_{X}E_A)$. 
\end{definition}

	In the $\ast$-isomorphic sense, the definition of equivariant Roe algebra with coefficients does not depend on the choice of $\Gamma$-invariant dense subsets $Z_X$ (see \cite[Section 5.2]{WillettYu-Book}).

\begin{remark}
The Roe algebra was first introduced by Roe in \cite{Roe1988}, whose $K$-theory provides a receptacle for higher indices of elliptic differential operators on open manifolds.
\end{remark}

\begin{definition}\label{Def-equi-coarsemap}
	Let $X$ and $Y$ be two $\Gamma$-spaces. A map $f: X\rightarrow Y$ is called to be an \textit{equivariant coarse map}, if 
	  \begin{itemize}
	  	\item $f(x\gamma)=f(x)\gamma$ for all $x\in X$ and $\gamma\in \Gamma$;
	  	\item $f^{-1}(K)$ is pre-compact for any compact subset $K\subseteq Y$;
	  	\item for any $R>0$, there exists $S>0$ such that $d(f(x), f(x'))\leq S$ for any $x,x'\in X$ with $d(x,x')\leq R$.
	  \end{itemize}
	  
	  $X$ and $Y$ are called to be \textit{equivariantly coarsely equivalent}, if there exist two equivariant coarse maps $f: X\rightarrow Y$, $g: Y\rightarrow X$ and a constant $C>0$ such that
	  $$\max\{d(gf(x), x), d(fg(y), y)\}\leq C,$$
	  for all $x\in X$ and $y\in Y$.  
\end{definition}

\begin{lemma}\label{Lem-coveringmap}\cite[Theorem 5.2.6]{WillettYu-Book}
	Let $X$, $Y$ be two $\Gamma$-spaces and $A$ be a $\Gamma$-$C^{\ast}$-algebra. Assume that $f: X\rightarrow Y$ be an equivariant coarse map. Then $f$ leads to a $\ast$-homomorphism from $C^{\ast}(X, A)^{\Gamma}$ to $C^{\ast}(Y, A)^{\Gamma}$, and hence induces a homomorphism
	  $$f_{\ast}: K_{\ast}(C^{\ast}(X, A)^{\Gamma}) \rightarrow K_{\ast}(C^{\ast}(Y, A)^{\Gamma}).$$
\end{lemma}

\begin{corollary}\label{Cor-coarseequi-K}
	If $X$ is equivariantly coarsely equivalent to $Y$, then $C^{\ast}(X, A)^{\Gamma}$ is naturally isomorphic to $C^{\ast}(Y, A)^{\Gamma}$.
\end{corollary}

Now, we consider the relationship between the equivariant Roe algebras with coefficients and the reduced crossed product. Let $X$ be a $\Gamma$-space and $A$ be a $\Gamma$-$C^{\ast}$-algebra. Since the action of $\Gamma$ on $X$ is co-compact, the map $\Gamma \rightarrow X$ defined by $\gamma \mapsto x_0 \gamma$ is an equivariantly coarsely equivalent map, where $x_0\in X$ is a fixed point. Hence, by Corollary \ref{Cor-coarseequi-K}, $C^{\ast}(X, A)^{\Gamma}$ is naturally $\ast$-isomorphic to $C^{\ast}(\Gamma, A)^{\Gamma}$ which naturally $\ast$-isomorphic to $(A \rtimes_r \Gamma) \otimes \mathcal{K}(H)$. Therefore, we have the following lemma.

\begin{lemma}(\cite[Lemma B.12]{GWY-2024} or \cite[Theorem 5.3.2]{WillettYu-Book})\label{Lem-eqRoe-crossprod}
	Let $X$ be a $\Gamma$-space and $A$ be a $\Gamma$-$C^{\ast}$-algebra. Then $C^{\ast}(X, A)^{\Gamma}$ is naturally $\ast$-isomorphic to $(A \rtimes_r \Gamma) \otimes \mathcal{K}(H)$.
\end{lemma}  

\subsection{Equivariant localization algebras with coefficients}

\begin{definition}\label{Def-localg}
	Let $X$ be a $\Gamma$-space and $A$ be a $\Gamma$-$C^{\ast}$-algebra. The \textit{equivariant localization algebra with coefficients} in $A$ of $X$, denoted by $C^{\ast}_{L}(X, A)^{\Gamma}$, is define to be the norm closure of the $\ast$-algebra consisting of all bounded and uniformly continuous functions $u: [0,\infty) \rightarrow C^{\ast}(X, A)^{\Gamma}$ such that
	  $$\lim_{t\rightarrow \infty}\prop(u(t))=0.$$
\end{definition}

\begin{remark}
	The localization algebra was first introduced by Yu in \cite{Yu-Localizationalg} whose $K$-theory provides an analytic model for $K$-homology (cf. \cite{QiaoRoe}\cite[Chapter 6]{WillettYu-Book}).
\end{remark}

Then, we have the following lemma, please refer to \cite[Lemma 3.4]{Yu-Localizationalg} or \cite[Theorem 6.6.3]{WillettYu-Book} for the proof.
\begin{lemma}\label{Lem-conticovmap}
	Let $X$, $Y$ be two $\Gamma$-spaces and $A$ be a $\Gamma$-$C^{\ast}$-algebra. Assume that $f: X \rightarrow Y$ be a uniformly continuous, equivariant coarse map. Then $f$ induces a homomorphism
	  $$f_{\ast}: K_{\ast}(C^{\ast}_{L}(X,A)^{\Gamma}) \rightarrow K_{\ast}(C^{\ast}_{L}(Y, A)^{\Gamma}).$$
\end{lemma}

We have the following two homological properties for the $K$-theory of localization algebras. Please refer to \cite[Theorem 6.3.4, Theorem 6.4.16 and Proposition 6.6.2]{WillettYu-Book} for their proofs.
\begin{definition}\label{Def-homotopy}
	Let $f^{(0)}, f^{(1)}: X\rightarrow Y$ be two uniformly continuous, equivariant coarse maps. $f^{(0)}$ is called to be \textit{equivariantly homotopic} to $f^{(1)}$, if there exists a uniformly continuous, equivariant coarse map $h: X\times [0,1]\rightarrow Y$ such that $h(x,0)=f^{(0)}(x)$ and $h(x,1)=f^{(1)}(x)$ for all $x\in X$.
	
	$X$ is called to be \textit{equivariantly homotopy equivalent} to $Y$, if there exist two uniformly continuous, equivariant coarse maps $f_1: X\rightarrow Y$ and $f_2: Y\rightarrow X$ such that $f_2f_1$ and $f_1f_2$ are equivariantly homotopic to $id_{X}$ and $id_{Y}$, respectively. 
\end{definition}
\begin{lemma}\label{Lem-Liphtp}
	Let $X$, $Y$ and $A$ be as above. If $X$ is equivariantly homotopy equivalent to $Y$, then $K_{\ast}(C^{\ast}_{L}(X, A)^{\Gamma})$ is naturally isomorphic to $K_{\ast}(C^{\ast}_{L}(Y, A)^{\Gamma})$.
\end{lemma}

\begin{lemma}\label{Lem-equiMV}
	If $\{X_1, X_2\}$ is a cover of $X$ by $\Gamma$-invariant closed subsets, then the following Mayer-Vietoris six-term exact sequence holds
	$$\small\xymatrix{
		K_0(C^{\ast}_{L}(X_1\cap X_2, A)^{\Gamma}) \ar[r] & 
		K_0(C^{\ast}_{L}(X_1, A)^{\Gamma})\oplus K_0(C^{\ast}_{L}(X_2, A)^{\Gamma}) \ar[r] & 
		K_0(C^{\ast}_{L}(X, A)^{\Gamma})\ar[d] \\
		K_1(C^{\ast}_{L}(X, A)^{\Gamma})\ar[u] &
		K_1(C^{\ast}_{L}(X_1, A)^{\Gamma})\oplus K_1(C^{\ast}_{L}(X_2, A)^{\Gamma})\ar[l] &
		K_1(C^{\ast}_{L}(X_1\cap X_2, A)^{\Gamma}).\ar[l] 
	}$$
\end{lemma}

\subsection{The Baum-Connes conjecture with coefficients}
\begin{definition}\label{Def-length}
	Let $\Gamma$ be a discrete countable group. A \textit{length function} on $\Gamma$ is a function $|\cdot|: \Gamma \rightarrow \mathbb{N}$ such that
	\begin{itemize}
		\item $|\gamma|=0$ if and only if $\gamma$ is the identity element;
		\item $|\gamma^{-1}|=|\gamma|$ for any $\gamma\in \Gamma$;
		\item $|\gamma_1 \gamma_2|\leq |\gamma_1|+|\gamma_2|$ for any $\gamma_1, \gamma_2\in \Gamma$.
	\end{itemize}
	A length function $|\cdot|$ on $\Gamma$ is called to be \textit{proper}, if $\{\gamma: |\gamma|\leq R\}$ is finite for any $R\geq 0$.
\end{definition}

There always exists a proper length function on any discrete countable group (see, for example, \cite[Proposition 1.2.2]{Nowak-Yu-Book}). Moreover, given a length function $|\cdot|$ on $\Gamma$, we can endow it with a right $\Gamma$-invariant metric defined by $d(\gamma_1, \gamma_2)=|\gamma_1 \gamma^{-1}_2|$. 

Let $\Gamma$ be a discrete countable group with a length function $|\cdot|$. For each $d\geq 0$, the \textit{Rips complex} of $\Gamma$ at scale $d$ denoted by $P_{d}(\Gamma)$, is a simplicial complex whose set of vertices is $\Gamma$ and a subset $\{\gamma_0, \cdots, \gamma_n\}$ spans an $n$-simplex if and only if $|\gamma_i \gamma^{-1}_j|\leq d$ for any $i,j=0, \cdots, n$. 

Let $d_{S_d}$ be a path metric on $P_d(\Gamma)$ whose restriction to each simplex is the standard spherical metric on the unit sphere by mapping $\sum_{i=0}^n t_i\gamma_i$ to 
$$\left(t_0 \Big/ \sqrt{\sum t^2_i}, \cdots, t_n \Big/\sqrt{\sum t^2_i}\right).$$ 
Then we define a metric $d_{P_{d}}$ on $P_d(\Gamma)$ to be 
$$d_{P_{d}}(x,x')=\inf\left\{\sum_{i=0}^n d_{S_d}(\gamma_i, \gamma'_i)+\sum_{i=0}^{n-1} |\gamma'_i\gamma^{-1}_{i+1}|\right\},$$
for all $x,x'\in P_{d}(\Gamma)$, where the infimum is taken over all sequences of the form $x=\gamma_0, \gamma'_0, \gamma_1, \gamma'_1, \cdots, \gamma_n, \gamma'_n=y$ with $\gamma_1, \cdots, \gamma_n, \gamma'_0,\cdots, \gamma'_{n-1}\in \Gamma$. Then $P_{d}(\Gamma)$ is a $\Gamma$-space equipped with a right $\Gamma$-action given by 
$$\gamma: \sum_{i=0}^n t_i \gamma_i \mapsto \sum_{i=0}^n t_i \gamma_i \gamma,$$
for any $\gamma, \gamma_0, \cdots, \gamma_n\in \Gamma$. 

For $d_1\leq d_2$, the inclusion map $i_{d_1d_2}$ induces the following two homomorphisms by Lemma \ref{Lem-coveringmap} and Lemma \ref{Lem-conticovmap}:
$$i_{d_1d_2, \ast}: K_{\ast}(C^{\ast}(P_{d_1}(\Gamma), A)^{\Gamma})\rightarrow K_{\ast}(C^{\ast}(P_{d_2}(\Gamma), A)^{\Gamma});$$
$$i_{d_1d_2, \ast}: K_{\ast}(C_L^{\ast}(P_{d_1}(\Gamma), A)^{\Gamma})\rightarrow K_{\ast}(C_L^{\ast}(P_{d_2}(\Gamma), A)^{\Gamma}).$$

Moreover, the evaluation at zero map:
$$e: C^{\ast}_{L}(P_{d}(\Gamma), A)^{\Gamma}\rightarrow C^{\ast}(P_{d}(\Gamma), A)^{\Gamma}, \:u\mapsto u(0),$$
induces a homomorphism:
$$e_{\ast}: K_{\ast}(C^{\ast}_{L}(P_{d}(\Gamma), A)^{\Gamma}) \rightarrow K_{\ast}(C^{\ast}(P_{d}(\Gamma), A)^{\Gamma}),$$
that satisfies $i_{d_1d_2, \ast}\circ e_{\ast}=e_{\ast} \circ i_{d_1d_2, \ast}$ for all $d_1\leq d_2$. 

Then, the \textit{Baum-Connes conjecture with coefficients} in $A$ for $\Gamma$ states as below.

\begin{conjecture}\label{BCC}
	Let $\Gamma$ be a discrete countable group and $A$ be a $\Gamma$-$C^{\ast}$-algebra. Then the following homomorphism 
	$$e_{\ast}: \lim_{d\rightarrow \infty} K_{\ast}(C^{\ast}_{L}(P_{d}(\Gamma), A)^{\Gamma}) \rightarrow \lim_{d\rightarrow \infty} K_{\ast}(C^{\ast}(P_{d}(\Gamma), A)^{\Gamma}),$$
	is an isomorphism.
\end{conjecture}

In particular, when $A=\mathbb{C}$, the above conjecture is called the \textit{Baum-Connes conjecture}.

The original statement of the Baum-Connes conjecture with coefficients is by employing the equivariant $KK$-theory to compute the $K$-theory of reduced crossed products, namely, the Baum-Connes assembly map 
$$\mu: \lim_{d\rightarrow \infty} KK^{\Gamma}_{\ast}(P_{d}(\Gamma), A) \rightarrow K_{\ast}(A \rtimes_{r} \Gamma)$$
is an isomorphism. Here the Rips complex provides a concrete model for the classifying space $\underline{E}\Gamma$ for proper actions of $\Gamma$. Thanks to Guentner, Willett and Yu's work \cite[Appendix B]{GWY-2024}, we can identify the above original statement of the Baum-Connes conjecture with coefficients with Conjecture \ref{BCC}.

\begin{proposition}\cite[Theorem B.3]{GWY-2024}\label{Prop-BCC-Orig}
	Let $\Gamma$ be a discrete countable group and $A$ be a $\Gamma$-$C^{\ast}$-algebra. Then we have the following natural commutative diagram
	 $$\xymatrix{
	 	KK^{\Gamma}_{\ast}(P_{d}(\Gamma), A) \ar[r]^{\mu} \ar[d] & 
	 	K_{\ast}(A \rtimes_{r} \Gamma) \ar[d] \\
	    K_{\ast}(C^{\ast}_{L}(P_{d}(\Gamma), A)^{\Gamma}) \ar[r]^{e_{\ast}} &
	 	K_{\ast}(C^{\ast}(P_{d}(\Gamma), A)^{\Gamma})
	 }$$
	 and the above two vertical maps are all isomorphisms for any $d\geq 0$. Hence, the Baum-Connes assembly map $\mu$ is an isomorphism if and only if Conjecture \ref{BCC} holds. 
\end{proposition}

\section{The Baum-Connes conjecture with coefficients for direct products}\label{Sec-BCprod}

In this section, we will discuss the Baum-Connes conjecture with coefficients for the direct products of groups. Firstly, let us fix some notations for this section. Let $\Gamma$, $G$ be two discrete countable groups and $A$ be a $\Gamma\times G$-$C^{\ast}$-algebra. Let $X$ be a $\Gamma$-space with a $\Gamma$-invariant countable dense subset $Z_X$ and $Y$ be a $G$-space with a $G$-invariant countable dense subset $Z_Y$. And let $H$ be a separable Hilbert space.

\subsection{The direct product of groups}

We define a metric on the direct product space $P_{d}(\Gamma)\times P_{d}(G)$ of Rips complexes at scale $d$ by 
$$d((x, y),(x',y'))=\max\{d_{P_{d}}(x,x'), d_{P_{d}}(y, y')\},$$
for all $x,x'\in P_{d}(\Gamma)$ and $y,y'\in P_{d}(G)$. Then $P_{d}(\Gamma)\times P_{d}(G)$ is a $\Gamma\times G$-space equipped with a $\Gamma\times G$-action defined by 
$$(x,y)(\gamma, g)=(x\gamma, yg),$$
for any $\gamma\in \Gamma$, $g\in G$ and $x\in P_{d}(\Gamma)$, $y\in P_{d}(G)$.


Then, we have the following lemma.
\begin{lemma}\cite[Lemma 4.16]{Zhang-CBCFC}\label{Lem-reduction-product}
	For any $d\geq 0$, the Rips complex $P_{d}(\Gamma\times G)$ is equivariantly homotopy equivalent to and equivariantly coarsely equivalent to $P_{d}(\Gamma)\times P_{d}(G)$.
\end{lemma}

Combining Lemma \ref{Lem-reduction-product} with Corollary \ref{Cor-coarseequi-K} and Lemma \ref{Lem-Liphtp}, we have the following lemma.

\begin{lemma}\label{Prop-reduction-BC}
	Let $\Gamma$, $G$ be two discrete countable groups and $A$ be a $\Gamma\times G$-$C^{\ast}$-algebra. Then the Baum-Connes conjecture with coefficients in $A$ holds for $\Gamma\times G$ if and only if the following evaluation at zero homomorphism
	$$e_{\ast}: \lim_{d\rightarrow \infty} K_{\ast}\left(C^{\ast}_{L}\left(P_{d}(\Gamma)\times P_{d}(G), A\right)^{\Gamma\times G}\right) \rightarrow \lim_{d\rightarrow \infty} K_{\ast}\left(C^{\ast}\left(P_{d}(\Gamma)\times P_{d}(G), A\right)^{\Gamma \times G}\right),$$
	is an isomorphism.
\end{lemma}

\subsection{Equivariant Roe algebras for direct products}

Let $\alpha$ be an action by $\Gamma\times G$ on $A$. Then $\alpha$ induces two actions $\alpha^{\Gamma}$ by $\Gamma$ and $\alpha^{G}$ by $G$ on $A$ defined as follows:
$$\alpha^{\Gamma}_{\gamma}(a)=\alpha_{(\gamma, e_{G})}(a);$$
$$\alpha^{G}_{g}(a)=\alpha_{(e_{\Gamma}, g)}(a),$$
for all $\gamma\in \Gamma$, $g\in G$ and $a\in A$, where $e_{G}$ and $e_{\Gamma}$ are the identity elements of $G$ and $\Gamma$, respectively. And we have $\alpha_{(\gamma, g)}=\alpha^{\Gamma}_{\gamma}\circ \alpha^{G}_{g}=\alpha^{G}_{g}\circ \alpha^{\Gamma}_{\gamma}$ for any $\gamma\in \Gamma$, $g\in G$. Consequently, we obtain an equivariant Roe algebra $C^{\ast}(Y, A)^{G}$. Moreover, we have an action $\beta$ by $\Gamma$ on $C^{\ast}(Y, A)^{G}$ defined as follows:
$$\beta_{\gamma}\left(\left( T_{y,y'} \right)_{y,y'\in Z_Y}\right)=\left(\left(\alpha^{\Gamma}_{\gamma}\otimes I\right)\left(T_{y,y'}\right)\right)_{y,y'\in Z_Y},$$
for all $\gamma\in \Gamma$ and $\left(T_{y,y'}\right)_{y,y'\in Z_Y} \in C^{\ast}(Y, A)^{G}$, where $\alpha^{\Gamma}_{\gamma}\times I: A\otimes \mathcal{K}(H)\otimes \mathcal{K}(\ell^2(G)) \rightarrow A\otimes \mathcal{K}(H)\otimes \mathcal{K}(\ell^2(G))$ maps $a\otimes K \otimes Q$ to $\alpha^{\Gamma}_{\gamma}(a)\otimes K\otimes Q$. Thus, we also obtain an equivariant Roe algebra $C^{\ast}(X, C^{\ast}(Y, A)^{G})^{\Gamma}$.

Now, we consider the relationship between $C^{\ast}(X\times Y, A)^{\Gamma\times G}$ and $C^{\ast}(X, C^{\ast}(Y, A)^{G})^{\Gamma}$. Firstly, we define
\begin{equation*}
  \begin{split}
	\mathbb{C}[X, \mathbb{C}[Y, A]^{G}]^{\Gamma}=&\big \{T\in \mathbb{C}[X, C^{\ast}(Y, A)^{G}]^{\Gamma}: \chi_KT, T\chi_K \in \mathcal{K}\left(\ell^2(Z_X)\otimes H\otimes \ell^2(\Gamma)\right)\otimes \mathbb{C}[Y, A]^{G} \\ 
	&\text{for any compact subset $K\subseteq X$} \big \}.
  \end{split}
\end{equation*}

\begin{lemma}\label{Lem-red-filtration} 
	Let $X$, $Y$, $A$ and $\Gamma$, $G$ be as above. Then the equivariant Roe algebra $C^{\ast}(X, C^{\ast}(Y, A)^{G})^{\Gamma}$ is identical to the norm closure of the $\ast$-algebra $\mathbb{C}[X, \mathbb{C}[Y, A]^{G}]^{\Gamma}$.
\end{lemma}

\begin{proof}
	Since the action of $\Gamma$ on $X$ is co-compact, thus there exists a compact subset $D\subseteq X$ such that $X=\cup_{\gamma\in \Gamma} D\gamma$.
	For any $T\in \mathbb{C}[X, C^{\ast}(Y, A)^{G}]^{\Gamma}$, since that $T$ is $\Gamma$-invariant with finite propagation and the action of $\Gamma$ on $X$ is proper, thus there exist $\gamma'_1, \cdots, \gamma'_n\in \Gamma$ such that
	$$T=T_{\gamma'_1}+\cdots+T_{\gamma'_n},$$
	and every $T_{\gamma'_i}$ is a $\Gamma$-invariant and locally compact operator with supported in $\cup_{\gamma\in \Gamma}(D\gamma'_i\gamma \times D\gamma)$. 
	Hence $T_{\gamma'_i}$ can be approximated by a sequence of operators in $\mathbb{C}[X, \mathbb{C}[Y, A]^{G}]^{\Gamma}$ because that $D$ is compact. Therefore, $T$ lies in the norm closure of $\mathbb{C}[X, \mathbb{C}[Y, A]^{G}]^{\Gamma}$, which implies $C^{\ast}(X, C^{\ast}(Y, A)^{G})^{\Gamma}$ is contained in the norm closure of $\mathbb{C}[X, \mathbb{C}[Y, A]^{G}]^{\Gamma}$. And the inverse inclusion is obvious. Thus, the proof is completed.
\end{proof}

\begin{proposition}\label{Lem-different-module}
    Let $X$, $Y$, $A$ and $\Gamma$, $G$ be as above. Then the equivariant Roe algebra $C^{\ast}(X, C^{\ast}(Y, A)^{G})^{\Gamma}$ is naturally (corresponding to $X$) $\ast$-isomorphic to the equivariant Roe algebra $C^{\ast}(X\times Y, A)^{\Gamma\times G}$.
\end{proposition}

\begin{proof}
    The algebras $C^{\ast}(X, C^{\ast}(Y, A)^{G})^{\Gamma}$ and $C^{\ast}(X\times Y, A)^{\Gamma\times G}$ act on $_{X}E_{C^{\ast}(Y, A)^G}=\ell^2(Z_X)\otimes C^{\ast}(Y, A)^G\otimes H\otimes \ell^2(\Gamma)$ and $_{X\times Y}{E}_{A}=\ell^2(Z_X)\otimes \ell^2(Z_Y)\otimes A\otimes H\otimes \ell^2(G)\otimes \ell^2(\Gamma)$, respectively. Recall that the right Hilbert $A$-module $_{Y}E_A=\ell^2(Z_Y)\otimes A\otimes H\otimes \ell^2(G)$. 
    
    Let $\iota$ be the embedding map from $C^{\ast}(Y, A)^{G}$ to $\mathcal{L}(_{Y}E_{A})$. Then we have an interior tensor product of Hilbert module (see \cite[Chapter 4]{Lance-HilbertMod})
    $$_{X}E_{C^{\ast}(Y, A)^G} \otimes_{\iota} {_{Y}}E_{A}.$$
    Define a linear operator $U: {_{X}}E_{C^{\ast}(Y, A)^G} \otimes_{\iota} {_{Y}E_{A}} \rightarrow \ell^2(Z_X)\otimes {_{Y}}E_{A} \otimes H \otimes \ell^2(\Gamma)$ by
    $$U\left(\left(\xi\otimes b\otimes h\otimes \eta\right)\otimes \vartheta\right)=\xi\otimes \iota(b)\vartheta \otimes h\otimes \eta,$$
    for all $\xi\otimes b\otimes h\otimes \eta\in {_{X}E_{C^{\ast}(Y, A)^G}}$ and $\vartheta \in {_{Y}E_{A}}$. Then $U$ is an isometric, surjective linear map, and hence is a unitary (\cite[Theorem 3.5]{Lance-HilbertMod}). Let $U_0: H\otimes H \rightarrow H$ be a unitary, then $V=I\otimes U_0$ is a unitary from $\ell^2(Z_X)\otimes {_{Y}}E_{A} \otimes H \otimes \ell^2(\Gamma)$ to ${_{X\times Y}E_{A}}$.
    
    Define a map $\psi: \mathcal{L}({_{X}E_{C^{\ast}(Y, A)^{G}}}) \rightarrow \mathcal{L}({_{X\times Y}E_{A}})$ by
    \begin{equation}\label{eq1}
    \psi(T)=VU(T\otimes I_{_{Y}E_{A}})U^{\ast}V^{\ast},
    \end{equation}
    for all $T\in \mathcal{L}({_{X}E_{C^{\ast}(Y, A)^{G}}})$. Then $\psi$ is an injective $\ast$-homomorphism because $\iota$ is injective. Thus, $\psi$ is an isometric linear map. 
    
    Now, we prove $\psi$ is surjective. For any operator $S\in \mathbb{C}[X\times Y, A]^{\Gamma\times G}$, we can represent $S$ as a matrix $\left(\left(S_{(x,y),(x',y')}\right)_{y,y'\in Z_Y}\right)_{x,x'\in Z_X}$, where $\left(S_{(x,y),(x',y')}\right)_{y,y'\in Z_Y} \in \mathbb{C}[Y, A]^G\otimes \ell^2(\Gamma)$. Let 
    $$T=\left(\left((I\otimes U^{\ast}_0)S_{(x,y),(x',y')}(I\otimes U_0)\right)_{y,y'\in Z_Y}\right)_{x,x'\in Z_X},$$ 
    then $T\in \mathbb{C}[X, \mathbb{C}[Y, A]^G]^{\Gamma}$ and $\psi(T)=S$, which implies $\psi$ is a surjective map from $\mathbb{C}[X, \mathbb{C}[Y, A]^G]^{\Gamma}$ to $\mathbb{C}[X\times Y, A]^{\Gamma\times G}$. Therefore, $\psi$ is a $\ast$-isomorphism from $C^{\ast}(X, C^{\ast}(Y, A)^{G})^{\Gamma}$ to $C^{\ast}(X\times Y, A)^{\Gamma\times G}$ by Lemma \ref{Lem-red-filtration}. Obviously, $\psi$ is natural for $X$. 
\end{proof}

\subsection{Equivariant localization algebras along one direction and results}
Consider a right Hilbert $A$-module
$$_{X\times Y} E_{A}=\ell^2(Z_X)\otimes \ell^2(Z_Y)\otimes A \otimes H \otimes \ell^2(\Gamma)\otimes \ell^2(G).$$

\begin{definition}\label{Def-prop-along}
	For $T\in \mathcal{L}(_{X\times Y} E_A)$, the \textit{propagation along $X$} of $T$, denoted by $\prop_{X}(T)$, is defined to be
	$$\prop_{X}(T)=\sup\{d(x,x'): \:\: \text{there exist $y,y'\in Y$ such that $\left(\left(x,y\right), \left(x',y'\right)\right)\in \supp(T)$}\}.$$
\end{definition}

\begin{definition}\label{Def-Localg-along}
	The \textit{equivariant localization algebra along $X$ with coefficients} in $A$ of $X\times Y$, denoted by $C^{\ast}_{L, X}(X\times Y, A)^{\Gamma\times G}$, is defined to be the norm closure of the $\ast$-algebra consisting of all bounded and uniformly continuous functions $u:[0,\infty)\rightarrow C^{\ast}(X\times Y, A)^{\Gamma\times G}$ such that
	$$\lim_{t\rightarrow \infty} \prop_{X}(u(t))=0.$$
\end{definition}

By Proposition \ref{Lem-different-module}, there exists a natural $\ast$-isomorphism $\psi: C^{\ast}(X, C^{\ast}(Y, A)^{G})^{\Gamma} \rightarrow C^{\ast}(X\times Y, A)^{\Gamma\times G}$ (see equation (\ref{eq1})). Then we get a $\ast$-isomorphism 
$$\tilde{\psi}: C^{\ast}_{L}(X, C^{\ast}(Y, A)^{G})^{\Gamma} \rightarrow C^{\ast}_{L, X}(X\times Y, A)^{\Gamma\times G},$$ 
defined by $\widetilde{\psi}(u)(t)=\psi(u(t))$ for all $u\in C^{\ast}_{L}(X, C^{\ast}(Y, A)^{G})^{\Gamma}$. Thus, we have the following lemma connecting the equivariant localization algebra along $X$ with coefficients and the equivariant localization algebra with different coefficients.

\begin{lemma}\label{Localong-Loc}
	Let $X$, $Y$, $A$ and $\Gamma$, $G$ be as above. Then $C^{\ast}_{L}(X, C^{\ast}(Y, A)^{G})^{\Gamma}$ is naturally (corresponding to $X$) $\ast$-isomorphic to $C^{\ast}_{L, X}(X\times Y, A)^{\Gamma\times G}$.
\end{lemma}

Combining Lemma \ref{Localong-Loc} with Lemma \ref{Lem-Liphtp} and Lemma \ref{Lem-equiMV}, we obtain the following two lemmas.
\begin{lemma}\label{Lem-Liphtp-Loc}
	 If $X$ is equivariantly homotopy equivalent to $X'$, then $K_{\ast}(C^{\ast}_{L,X}(X\times Y, A)^{\Gamma\times G})$ is naturally (corresponding to $X$ and $X'$) isomorphic to $K_{\ast}(C^{\ast}_{L, X'}(X'\times Y, A)^{\Gamma\times G})$.
\end{lemma}

\begin{lemma}\label{Lem-MV-Loc}
	If $\{X_1, X_2\}$ is a cover of $X$ by $\Gamma$-invariant closed subsets, then the following Mayer-Vietoris six-term exact sequence holds
	$$\footnotesize\xymatrix{
		K_0(C^{\ast}_{L, X_{1,2}}(X_{1,2}\times Y, A)^{\Omega}) \ar[r] & 
		\oplus_{i=1}^{2}K_0(C^{\ast}_{L, X_i}(X_i\times Y, A)^{\Omega}) \ar[r] & 
		K_0(C^{\ast}_{L, X}(X\times Y, A)^{\Omega})\ar[d] \\
		K_1(C^{\ast}_{L, X}(X\times Y, A)^{\Omega})\ar[u] &
		\oplus_{i=1}^{2}K_1(C^{\ast}_{L, X_i}(X_i\times Y, A)^{\Omega})\ar[l] &
		K_1(C^{\ast}_{L, X_{1,2}}(X_{1,2}\times Y, A)^{\Omega}),\ar[l] 
	}$$
	where $X_{1,2}=X_1\cap X_2$ and $\Omega=\Gamma\times G$.
\end{lemma}

Consider the evaluation at zero map
$$e: C^{\ast}_{L, X}(X\times Y, A)^{\Gamma\times G} \rightarrow C^{\ast}(X\times Y, A)^{\Gamma\times G}, \:\:e\mapsto e(0).$$
Then, we have the following proposition.

\begin{proposition}\label{Prop-BCalong-BC}
	Let $\Gamma$, $G$ be two discrete countable groups and $A$ be a $\Gamma\times G$-$C^{\ast}$-algebra. If the Baum-Connes conjecture with coefficients in $A \rtimes_r G$ holds for $\Gamma$.
	Then for any $d\geq 0$, the following homomorphism
	$$e_{\ast}: \lim_{k\rightarrow \infty} K_{\ast}\left(C^{\ast}_{L, P_{k}(\Gamma)}\left(P_{k}(\Gamma)\times P_{d}(G), A\right)^{\Gamma\times G}\right)
	 \rightarrow 
	\lim_{k\rightarrow \infty} K_{\ast}\left(C^{\ast}\left(P_{k}(\Gamma)\times P_{d}(G), A\right)^{\Gamma\times G}\right),$$
is an isomorphism.
\end{proposition}
\begin{proof}
	Let $\alpha$ be the $\Gamma \times G$-action on $A$. Then, the $\Gamma$-action $\beta$ on $A \rtimes_r G$ is defined to be $\beta_{\gamma}(\sum a_g g)= \sum \alpha_{(\gamma, e_{G})}(a_g) g$. Thus by Lemma \ref{Lem-eqRoe-crossprod}, $(A \rtimes_r G)\otimes \mathcal{K}(H)$ is $\Gamma$-equivariantly isomorphic to $C^{\ast}(P_d(G), A)^{G}$ for any $d\geq 0$, where the $\Gamma$-action on $\mathcal{K}(H)$ is trivial. Hence by the assumption, $\Gamma$ satisfies the Baum-Connes conjecture with coefficients in $C^{\ast}(P_d(G), A)^{G}$ for any $d\geq 0$, namely, we have the following isomorphism:
	$$\lim_{k\rightarrow \infty} K_{\ast}\left(C^{\ast}_{L}\left(P_{k}(\Gamma), C^{\ast}(P_d(G), A)^{G}\right)^{\Gamma}\right) \stackrel{\cong}{\longrightarrow} \lim_{k\rightarrow \infty} K_{\ast}\left(C^{\ast}(P_{k}(\Gamma), C^{\ast}(P_d(G), A)^{G})^{\Gamma}\right).$$
	Therefore, combing Proposition \ref{Lem-different-module} and Lemma \ref{Localong-Loc} with the above isomorphism, we completed the proof.
\end{proof}


Next, we will explore the connection between the localization algebras along one direction and the localization algebras of direct products. Firstly, let us recall an elementary lemma in the $K$-theory as below, please see \cite[Lemma 12.4.3]{WillettYu-Book} for its proof and generalization.

\begin{lemma}\label{Lem-stable-K}
	Let $B$ be a stable $C^{\ast}$-algebra (i.e., $B\cong B\otimes \mathcal{K}(H)$) and $C_{ub}([0,\infty), B)$ be the $C^{\ast}$-algebra of all bounded, uniformly continuous functions from $[0, \infty)$ to $B$. Then the evaluation at zero map 
	$$e: C_{ub}([0,\infty), B) \rightarrow B$$
	induces an isomorphism on the $K$-theory.
\end{lemma}

Any equivariant Roe algebra is a stable $C^{\ast}$-algebra by Lemma \ref{Lem-eqRoe-crossprod}.

Let $\iota: C^{\ast}_{L}(X\times Y, A)^{\Gamma\times G} \rightarrow C^{\ast}_{L, X}(X\times Y, A)^{\Gamma\times G}$ be the inclusion map.

\begin{proposition}\label{Lem-Key}
	Let $A$ be a $\Gamma\times G$-$C^{\ast}$-algebra. If one of the following conditions holds:
	\begin{enumerate}
		\item \label{Lem-Key-1} for any finite subgroup $F$ of $\Gamma$, the Baum-Connes conjecture with coefficients in $A_F$ holds for $G$, where $A_F$ is the $F$-fixed sub-algebra in $A$;
		\item \label{Lem-Key-2} for any finite subgroup $F$ of $\Gamma$, the Baum-Connes conjecture with coefficients in $A$ holds for $F\times G$. 
	\end{enumerate}
	Then for any $k\geq 0$,
	$$\iota_{\ast}: \lim_{d\rightarrow \infty} K_{\ast}(C^{\ast}_{L}(P_{k}(\Gamma)\times P_{d}(G), A)^{\Gamma\times G}) \rightarrow \lim_{d\rightarrow \infty} K_{\ast}(C^{\ast}_{L,P_{k}(\Gamma)}(P_{k}(\Gamma)\times P_{d}(G), A)^{\Gamma\times G})$$
	is an isomorphism.
\end{proposition}

\begin{proof}
	For any $k\geq 0$, since the metric on $\Gamma$ is proper, hence the dimension of $P_{k}(\Gamma)$ is finite. Let $P_{k}(\Gamma)^{(n)}$ be the $n$-dimensional skeleton of $P_{k}(\Gamma)$. We will prove the lemma by induction on $n$.
	
	For $n=0$, we have $P_{k}(\Gamma)^{(0)}=\Gamma$. Since $\Gamma$ is a proper metric space, thus there exists $c>0$ such that $d(\gamma,\gamma')>c$ for all $\gamma, \gamma'\in \Gamma$. Let $C^{\ast}_{L}(\prod_{\Gamma} P_{d}(G), A)^{\Gamma \times G}$ be a $C^{\ast}$-algebra consisting of all $u\in C^{\ast}_{L}(\Gamma \times P_{d}(G), A)^{\Gamma \times G}$ such that 
	$$u(t)_{(\gamma, y), (\gamma', y')}=0$$
    for any $\gamma\neq\gamma' \in \Gamma$ and $y,y'\in P_d(G)$, $t\geq 0$.
    Let 
    $$J: C^{\ast}_{L}(\prod_{\Gamma} P_{d}(G), A)^{\Gamma \times G} \rightarrow C^{\ast}_{L}(\Gamma\times P_{d}(G), A)^{\Gamma \times G}$$
    be the inclusion map. Then it induces an isomorphism $J_{\ast}$ on the $K$-theory level. Indeed, for any element $u\in C^{\ast}_{L}(\Gamma\times P_{d}(G), A)^{\Gamma \times G}$, there exists $t_0$ such that $\prop(u(t))<c$ for any $t\geq t_0$. Thus, the element $t\mapsto u(t+t_0)\in C^{\ast}_{L}(\prod_{\Gamma} P_{d}(G), A)^{\Gamma \times G}$, which implies $J_{\ast}$ is surjective. For injection of $J_{1}$, if $J_{1}([v])=0$ for an element $[v]\in K_1(C^{\ast}_{L}(\prod_{\Gamma} P_{d}(G), A)^{\Gamma \times G})$. Let $v^{(s)}$ be a path of unitaries connecting $v$ and $0$ in $M_{m}\left(C^{\ast}_{L}(\Gamma\times P_{d}(G), A)^{\Gamma \times G}\right)$. Choosing $v=v^{(0)}, v^{(1)}, \cdots, v^{(q)}=0$ such that $\|v^{(i)}-v^{(i-1)}\|<1$ for $i=1, \cdots, q$. Then there exists $t_1>0$ such that $\max_{0\leq i \leq q} \{\prop(v^{(i)}(t))\}<c$ for any $t\geq t_1$. Let $w^{(i)}(t)=v^{(i)}(t+t_1)$, then $w^{(i)}\in C^{\ast}_{L}(\prod_{\Gamma} P_{d}(G), A)^{\Gamma \times G}$ for any $i=0, \cdots, q$. Thus, $v$ is homotopic to $0$ by a sequence of linear paths between $w^{(i-1)}$ and $w^{(i)}$, which implies $[v]=0$ in $K_1(C^{\ast}_{L}(\prod_{\Gamma} P_{d}(G), A)^{\Gamma \times G})$. By a similar argument as above, we also have $J_0$ is injective. 
    
    Let $C^{\ast}_{L, \Gamma}(\prod_{\Gamma} P_{d}(G), A)^{\Gamma \times G}$ be a $C^{\ast}$-algebra consisting of all $u\in C^{\ast}_{L, \Gamma}(\Gamma \times P_{d}(G), A)^{\Gamma \times G}$ such that 
    $$u(t)_{(\gamma, y), (\gamma', y')}=0$$
    for any $\gamma\neq\gamma' \in \Gamma$ and $y,y'\in P_d(G)$, $t\geq 0$. Then using a similar argument as above paragraph, we can show that the inclusion map $J$ induces an isomorphism
    $$J_{\ast}: K_{\ast}\left( C^{\ast}_{L, \Gamma}(\prod_{\Gamma} P_{d}(G), A)^{\Gamma \times G} \right) \rightarrow K_{\ast}\left( C^{\ast}_{L, \Gamma}(\Gamma \times P_{d}(G), A)^{\Gamma \times G} \right).$$
    Thus, we have the following commutative diagram:
    $$\tiny\xymatrix{
    	K_{\ast}\left(C^{\ast}_{L}(P_d(G), A)^{G}\right) \ar[r]^{\cong} \ar[d]^{\iota_{\ast}} & 
    	K_{\ast}\left(C^{\ast}_{L}(\prod_{\Gamma} P_{d}(G), A)^{\Gamma \times G}\right)\ar[r]^{\cong}_{J_{\ast}} \ar[d]^{\iota_{\ast}} & 
    	K_{\ast}\left(C^{\ast}_{L}(\Gamma\times P_{d}(G), A)^{\Gamma \times G}\right)\ar[d]^{\iota_{\ast}}  \\
    	K_{\ast}\left( C_{ub}([0,\infty), C^{\ast}(P_d(G), A)^{G}) \right) \ar[r]^{\cong} &
    	K_{\ast}\left( C^{\ast}_{L, \Gamma}(\prod_{\Gamma} P_{d}(G), A)^{\Gamma \times G} \right)\ar[r]^{\cong}_{J_{\ast}} &
        K_{\ast}\left( C^{\ast}_{L, \Gamma}(\Gamma \times P_{d}(G), A)^{\Gamma \times G} \right),
    }
    $$
    where the left two horizontal arrows are actually isomorphic on the level of $C^{\ast}$-algebras.
    Moreover, we also have the following commutative diagram:
    $$\xymatrix{
        \lim_{d\rightarrow \infty} K_{\ast}\left(C^{\ast}_{L}(P_d(G), A)^{G}\right) \ar[r]^{e_{\ast}} \ar[d]^{\iota_{\ast}} &
        \lim_{d\rightarrow \infty} K_{\ast}\left(C^{\ast}(P_d(G), A)^{G}\right)\\
        \lim_{d\rightarrow \infty} K_{\ast}\left( C_{ub}([0,\infty), C^{\ast}(P_d(G), A)^{G}) \right) \ar[ur]_{e_{\ast}},
    }
    $$
    and two evaluation at zero map $e$ induce two above isomorphisms $e_{\ast}$ by the assumption and Lemma \ref{Lem-stable-K}. Thus, combing the above two commutative diagrams, we have that $\iota_{\ast}$ is an isomorphism for the case of $n=0$.
    
	Now, we assume the lemma holds for the case of $n=l-1$. Next, we will prove it for $n=l$. Let $c(\Delta)$ be the center of any $l$-dimensional simplex $\Delta$ in $P_{k}(\Gamma)$. Define 
	$$\Delta_1=\{x\in \Delta: d(x, c(\Delta))\leq 1/10\};\:\: \Delta_2=\{x\in \Delta: d(x, c(\Delta))\geq 1/10\}.$$ 
	And let
	$$X_1=\bigcup\{\Delta_1:\text{dim}(\Delta)=l\};\:\:X_2=\bigcup\{\Delta_2:\text{dim}(\Delta)=l\}.$$
	Then, $\{X_1, X_2\}$ is a cover of $P_{k}(\Gamma)^{l}$ by $\Gamma$-invariant closed subsets. Besides, $X_2$ and $X_1\cap X_2$ are equivariantly homotopy equivalent to $P_{k}(\Gamma)^{(l-1)}$ and the disjoint union of the boundaries of all $l$-dimensional simplices of $P_{k}(\Gamma)$, respectively. Thus, $\iota_{\ast}$ are isomorphic for $X_2$ and $X_1\cap X_2$ by the inductive assumption and Lemma \ref{Lem-Liphtp}, Lemma \ref{Lem-Liphtp-Loc}. 
	Moreover, $X_1$ is equivariantly homotopy equivalent to $\{c(\Delta): \text{dim}(\Delta)=l\}$. However, each $c(\Delta)$ is a $F$-fixed point for some finite subgroup $F$ of $\Gamma$. Thus if condition (\ref{Lem-Key-1}) or condition (\ref{Lem-Key-2}) holds, then by Lemma \ref{Lem-Liphtp} and Lemma \ref{Lem-Liphtp-Loc} again, $\iota_{\ast}$ is an isomorphism for $X_1$. Therefore, by Lemma \ref{Lem-equiMV}, Lemma \ref{Lem-MV-Loc} and the five lemma, $\iota_{\ast}$ is an isomorphism for the case of $n=l$. 
\end{proof}


Combining Proposition \ref{Lem-Key} and Proposition \ref{Prop-BCalong-BC} with Lemma \ref{Prop-reduction-BC}, we obtain the following main theorem in this section.

\begin{theorem}\label{main-thm1}
    Let $\Gamma$ and $G$ be two discrete countable groups, $A$ be a $\Gamma\times G$-$C^{\ast}$-algebra. If 
    \begin{enumerate}
    	\item \label{main-thm1-1} $F\times G$ satisfies the Baum-Connes conjecture with coefficients in $A$ for any finite subgroup $F$ of $\Gamma$.
    	\item \label{main-thm1-2} $\Gamma$ satisfies the Baum-Connes conjecture with coefficients in $A\rtimes_r G$.
    \end{enumerate} 
    Then, $\Gamma\times G$ satisfies the Baum-Connes conjecture with coefficients in $A$.
\end{theorem}

\begin{remark}
	In the first version of the paper, we replaced condition (\ref{main-thm1-1}) by a simplified condition stating that $G$ satisfies the Baum-Connes conjecture with coefficients in $A$. It is not true in generally since Meyer constructed an interesting counterexample for it in \cite{Meyer-BC-counterex}.
\end{remark}

  When $A=\mathbb{C}$, we get the following corollary.
\begin{corollary}\label{Cor-app-no-coeff}
	Let $\Gamma$ and $G$ be two discrete countable groups. If
	\begin{enumerate}
		\item $G$ satisfies the Baum-Connes conjecture;
		\item $\Gamma$ satisfies the Baum-Connes conjecture with coefficients in the reduced group $C^{\ast}$-algebra $C^{\ast}_{r}(G)$ (with the trivial $\Gamma$-action).
	\end{enumerate}
	Then, $\Gamma\times G$ satisfies the Baum-Connes conjecture.
\end{corollary}

\section{Main results}\label{Sec-main-thm}
In this section, we will study the Baum-Connes conjecture with coefficients for the group extension:
$$1\rightarrow N \rightarrow \Gamma \xrightarrow{q} \Gamma/ N \rightarrow 1.$$
Let $A$ be a $\Gamma$-$C^{\ast}$-algebra acted by an action $\alpha$ of $\Gamma$. Firstly, we need the following reduction result for the Baum-Connes conjecture with coefficients.

\begin{proposition}\label{Thm-BC-group-sub}\cite[Proposition 2.3]{CE-Permanence-BC}
	Let $G$ be a discrete countable group, $N'$ be a subgroup of $\Gamma$ and $B$ be a $G$-$C^{\ast}$-algebra. Then the Baum-Connes conjecture with coefficients in $B$ holds for $N'$ if and only if the Baum-Connes conjecture with coefficients in $C_0(G/ N', B)$ holds for $G$, where the action of $G$ on $C_0(G/ N', B)$ is defined by $g(f)(g'N')=g(f(g^{-1}g'N'))$ for any $f\in C_0(G/ N', B)$.
\end{proposition}



The group $\Gamma$ can be seen as a subgroup of $\Gamma \times \Gamma/ N$ by the following embedding:
$$\Gamma \rightarrow \Gamma \times \Gamma/ N, \: \gamma \mapsto (\gamma, [\gamma]).$$
Let $\beta'$ be an action of $\Gamma \times \Gamma/ N$ on $C_0(\Gamma/ N, A)$ defined by 
\begin{equation}\label{eq-product-action}
	\beta'_{(\gamma_1, [\gamma'_1])}(f)([\gamma_2])=\alpha_{\gamma_1}(f([\gamma^{-1}_1 \gamma_2 \gamma'_1])),
\end{equation}
for $f\in C_0(\Gamma/ N, A)$.
Thus, by Proposition \ref{Thm-BC-group-sub}, we have the following lemma.

\begin{lemma}\label{Lem-BC-sub-group}
	The following two statements are equivalent:
	\begin{enumerate}
		\item $\Gamma$ satisfies the Baum-Connes conjecture with coefficients in $A$;
		\item $\Gamma \times \Gamma/ N$ satisfies the Baum-Connes conjecture with coefficients in $C_0(\Gamma/ N, A)$.
	\end{enumerate}
\end{lemma}

Moreover, by applying Theorem \ref{main-thm1} to the $\Gamma \times \Gamma / N$-$C^{\ast}$-algebra $C_0(\Gamma/ N, A)$ equipped with the action $\beta'$ defined by (\ref{eq-product-action}), we have the following lemma.

\begin{lemma}\label{Lem-BC-quotient}
	If the following two conditions hold:
	\begin{enumerate}
		\item $\Gamma\times F$ satisfies the Baum-Connes conjecture with coefficients in $C_0(\Gamma / N, A)$ for any finite subgroup $F$ of $\Gamma / N$;
		\item $\Gamma / N$ satisfies the Baum-Connes conjecture with coefficients in $C_0(\Gamma / N, A) \rtimes_r \Gamma$.
	\end{enumerate}
	Then, the Baum-Connes conjecture with coefficients in $C_0(\Gamma / N, A)$ holds for $\Gamma \times \Gamma / N$.
\end{lemma}


Applying Proposition \ref{Thm-BC-group-sub} again to $\Gamma$ with a subgroup $q^{-1}(F)$ for any finite subgroup $F$ of $\Gamma / N$, we get the following lemma.

\begin{lemma}\label{Lem-BC-N-Gamma}
	Let $F$ be a finite subgroup of $\Gamma / N$. Then the following two statements are equivalent:
	\begin{enumerate}
		\item $q^{-1}(F)$ satisfies the Baum-Connes conjecture with coefficients in $A$;
		\item $\Gamma \times F$ satisfies the Baum-Connes conjecture with coefficients in $C_0(\Gamma / N, A)$.
	\end{enumerate}
\end{lemma}

Combing Lemma \ref{Lem-BC-sub-group}, Lemma \ref{Lem-BC-quotient} and Lemma \ref{Lem-BC-N-Gamma}, we obtain the following main theorem of this paper.

\begin{theorem}\label{main-thm}
	If the following two conditions hold:
	\begin{enumerate}
		\item the group $q^{-1}(F)$ satisfies the Baum-Connes conjecture with coefficients in $A$ for any finite subgroup $F$ of $\Gamma / N$;
		\item the group $\Gamma / N$ satisfies the Baum-Connes conjecture with coefficients in $C_0(\Gamma / N, A) \rtimes_r \Gamma$.
	\end{enumerate}
	Then the group $\Gamma$ satisfies the Baum-Connes conjecture with coefficients in $A$.
\end{theorem}

When $A=\mathbb{C}$, we have the following corollary.
\begin{corollary}\label{main-cor1}
	If the following two conditions hold:
	\begin{enumerate}
		\item the group $q^{-1}(F)$ satisfies the Baum-Connes conjecture for any finite subgroup $F$ of $\Gamma / N$;
		\item the group $\Gamma / N$ satisfies the Baum-Connes conjecture with coefficients in $C_0(\Gamma / N) \rtimes_r \Gamma$.
	\end{enumerate}
	Then the group $\Gamma$ satisfies the Baum-Connes conjecture.
\end{corollary}

\bibliographystyle{plain}
\bibliography{BCextensions}

\end{document}